\documentclass[12pt,letterpaper]{article}
\usepackage[utf8]{inputenc}
\usepackage{amsmath}
\usepackage{amsfonts}
\usepackage{amssymb}
\usepackage{hyperref}
\usepackage[left=2cm,right=2cm,top=2cm,bottom=2cm]{geometry}

\usepackage{graphicx}
\author{Daniel Tan}
\date{}
\title{$ C_n $-cofinite twisted modules for $ C_2 $-cofinite vertex operator algebras}

\DeclareMathOperator{\res}{Res} 
\DeclareMathOperator{\End}{End} 
\DeclareMathOperator{\one}{\mathbf{1}} 
\DeclareMathOperator{\id}{\text{id}} 
\newcommand{\iu}{\boldsymbol{i}} 
\DeclareMathOperator{\Span}{Span} 
\DeclareMathOperator{\wt}{wt} 
\DeclareMathOperator{\coln}{:} 
\newcommand{\w}{\widetilde{w}} 

\usepackage{amsthm}
\newtheorem{theorem}{Theorem}[section]
\newtheorem{proposition}[theorem]{Proposition}

\newtheorem{lemma}[theorem]{Lemma}

\theoremstyle{definition}
\newtheorem{definition}[theorem]{Definition}
\newtheorem{remark}[theorem]{Remark}

\usepackage[ backend=biber, sorting= nyt, style=alphabetic, url = false, isbn = false, doi = false, minalphanames=4, maxalphanames=10, maxbibnames=10]{biblatex}
\DeclareBibliographyCategory{cited}
\AtEveryCitekey{\addtocategory{cited}{\thefield{entrykey}}}

\numberwithin{equation}{section}

\begin{document}
\maketitle

\begin{abstract}
Given a vertex operator algebra $ V $ with a general automorphism $ g $ of $ V $, we introduce a notion of $ C_n $-cofiniteness for weak $ g $-twisted $ V $-modules. When $ V $ is $ C_2 $-cofinite and of CFT type, we show that all finitely generated weak $ g $-twisted $ V $-modules are $ C_n $-cofinite for all $ n \in \mathbb{Z}_{>0} $.
\end{abstract}

\paragraph{Keywords} Vertex operator algebras, twisted modules, $ C_n $-cofiniteness

\setcounter{section}{-1}

\section{Introduction}
A vertex operator algebra $ V $ is said to be \textit{$ C_2 $-cofinite} if its subspace 
\begin{equation*}
C_2(V) = \Span \{ \res_x x^{-2} Y(u,x)v : u, v \in V \}
\end{equation*}
has finite codimension in $ V $. The notion of $ C_2 $-cofiniteness was originally introduced by Zhu \cite{ZhuModInv} as a technical assumption (among other conditions) to prove the convergence and modular invariance of traces of products of vertex operators acting on irreducible $ V $-modules. This result was later strengthened in \cite{MiyamototC2} to remove certain assumptions, and extended in \cite{DLM_Modular} to certain $ g $-twisted modules; in both works, $ C_2 $-cofiniteness remained a key assumption. 

In \cite{Lifinite}, $ C_2 $-cofiniteness was generalized to the notion of \textit{$ C_n $-cofiniteness} for vertex operator algebras and their modules. Certain spanning sets for vertex operator algebras were obtained in \cite{GabNeit}, and used to show that a vertex operator algebra is $ C_n $-cofinite for all $ n \in \mathbb{Z}_{\geq 2} $, if it is $ C_2 $-cofinite. Buhl \cite{Buhl} extended this result to $ V $-modules (related results were also proved in \cite{AN} in the presence of an extra condition called ``$ B_1 $-finiteness"). More precisely, Buhl proved that the $ C_2 $-cofiniteness of a \textit{CFT-type} (or \textit{positive-energy}) vertex operator algebra $ V $ is sufficient to derive certain spanning sets for finitely generated $ V $-modules, hence proving their $ C_n $-cofiniteness for all $ n \in \mathbb{Z}_{\geq 2} $. These results were used in \cite{ABD} to prove a certain relationship between $ C_2 $-cofiniteness, rationality and regularity of $ V $ (the latter two being certain niceness properties for the representation theory of $ V $). 

Buhl's results imply that the finitely generated modules for a $ C_2 $-cofinite CFT-type vertex operator algebra are $ C_2 $-cofinite, and hence also $ C_1 $-cofinite. These properties have deep implications for the representation theory of vertex operator algebras. Huang \cite{HuangDEs} proved that genus-zero correlation functions defined by products of intertwining operators among $ C_1 $-cofinite modules converge as solutions to differential equations with regular singularities.  This, in turn, provides examples of vertex operator algebras with a natural braided tensor category structure on their category of modules \cite{HL1,HLZ1}. Furthermore, the $ C_2 $-cofiniteness of modules is an important assumption in proving the convergence and modular invariance of genus-one correlation functions \cite{HuangModularDEs}, the Verlinde conjecture \cite{Huang_VOAsVerlinde},  and the rigidity and modularity in the braided tensor categories of $ V $-modules \cite{HuangRigid}. Since analogs of these results are conjectured \cite{Reptheoryandorbifoldcft} in the study of orbifold conformal field theory, extending Buhl’s result to the twisted setting is a natural and important next step. 

Orbifold conformal field theories are conformal field theories constructed in a certain way from an existing conformal field theory and a group of its automorphisms. Frenkel, Lepowsky and Meurman's Moonshine Module \cite{FLM84, FLM} was the first example of such a construction. Motivated in part by the Moonshine Module construction, physicists Dixon, Harvey, Vafa and Witten in \cite{DHVW1, DHVW2} proposed, on a physical level of rigor, a procedure for constructing orbifold conformal field theories from ``twisted sectors" (i.e., twisted modules). While there are additional ingredients in this construction, such as twisted intertwining operators, we focus on the twisted modules in this paper (see \cite{Reptheoryandorbifoldcft} for an overview of the mathematical approach using the representation theory of vertex operator algebras).

Given an automorphism $ g $ of a vertex operator algebra $ V $, a \textit{weak $ g $-twisted $ V $-module} is a vector space equipped with a ``twisted vertex operator" action of $ V $, similar to that of a $ V $-module action, but twisted by $ g $ as the vertex operator circles the origin in the complex plane. This \textit{equivariance property} $ Y(u,x) = e^{2 \pi \iu x \frac{d}{dx}} Y(gu,x) $ is the fundamental property distinguishing weak $ g $-twisted $ V $-modules from weak untwisted $ V $-modules. We say \emph{weak} to indicate that no grading restrictions are imposed.

Twisted modules for finite-order automorphisms were first introduced in \cite{FFR, DongTwisted} by axiomatizing the structures appearing in \cite{LepowskyWilson, FLM84, FLM_E8, FLM_VOCalc, Lepowsky_VOandMonster, FLM} and Lepowsky's calculus of twisted vertex operators in \cite{LepowskyCalculus}. In \cite{Gentwistmod}, twisted modules for general automorphisms were introduced with an axiom known as \textit{duality}, serving as an analytic replacement for the algebraic \textit{Jacobi identity}.  In \cite{bakalovtwist}, a completely algebraic definition of a twisted module (for a vertex algebra) for a general automorphism was introduced. Bakalov's definition uses the notion of \textit{local fields} and an \textit{$ n $-th product} (inspired by \cite{LiLocalSystems,LepowskyLi}), which he showed is equivalent to a \textit{Borcherds identity} when the automorphism is locally finite. 

Jacobi identities for vertex operator algebras and their (twisted) modules are algebraic relations discovered in \cite{FLM}, of shape similar to the Jacobi identity for Lie algebras, featuring (twisted) vertex operators and certain formal $ \delta $-functions of three formal variables. In \cite{Huang2016AssociativeAF}, the duality axiom in \cite{Gentwistmod} was shown to produce a Jacobi identity equivalent to Bakalov's Borcherds identity. This Jacobi identity is analogous to the original ``twisted Jacobi identity" discovered in \cite{FLM}. We choose to use a Jacobi identity in our notion of a weak $ g $-twisted $ V $-module in Definition~\ref{def:weak g-twisted module} so it will be immediately clear that our notion relaxes the grading conditions of a $ g $-twisted $ V $-module in Definition 3.1 of \cite{Gentwistmod} with the Jacobi identity from \cite{Huang2016AssociativeAF}. Furthermore, our notion is clearly a generalization of a weak $ g $-twisted $ V $-module in Definition 3.1 of \cite{DLMtwisted}, permitting $ g $ to be a general automorphism of $ V $.  

Our work generalizes the results of \cite{Buhl} to weak $ g $-twisted $ V $-modules for a general automorphism $ g $ of $ V $. We introduce a notion of $ C_n $-cofiniteness for weak $ g $-twisted $ V $-modules. We do not impose any constraints (such as finite order or semisimplicity) on the automorphism $ g $. When $ V $ is $ C_2 $-cofinite, there is a finite set $ B \subseteq V $ such that $ V = \Span(B \cup \{ \one\}) + C_2(V) $. When $ V $ is also of CFT type, we show that each finitely generated weak $ g $-twisted $ V $-module $ W $ has a spanning set given by a list of modes of elements in $ B $ acting on the generators for $ W $ in a non-decreasing order of degree.  The modes can repeat a sufficiently high degree only finitely many times. This so-called ``PBW-type spanning set" originates from ideas in \cite{GabNeit}, \cite{Buhl} and \cite{ABD}. Our spanning set shows that every finitely generated weak $ g $-twisted $ V $-module $ W $ is $ C_n $-cofinite for all $ n \in \mathbb{Z}_{> 0} $. The method used to prove these propositions is adapted from the method in \cite{Buhl}.  

The results in this paper demonstrate that the notion of $ C_n $-cofiniteness naturally extends to twisted modules and remains implied by the $ C_2 $-cofiniteness of a CFT-type vertex operator algebra. We hope that the assumption of $ C_n $-cofiniteness may be applied in proving the convergence of products of twisted intertwining operators \cite{HUANG2018346} among modules twisted by elements in a finite group, analogous to the untwisted theory \cite{HuangDEs}. Such convergence plays an essential role in constructing the associator natural isomorphism \cite{Reptheoryandorbifoldcft} in the $ G $-crossed tensor category of twisted modules (see also \cite{TuraevCrossed,KirillovGEq}). 

Similar spanning sets for twisted modules as ours have been found in \cite{Yamauchi} (see Lemma~1 and Section~3). However, they do not apply their results to prove $ C_n $-cofiniteness, and their results are only for automorphisms of finite order, whereas our results are for general automorphisms of possibly infinite order.

In \cite{cao2025decreasingfiltrationsc2algebratwisted}, the same notion of $ C_n $-cofinite twisted modules was proposed when the order of $ g $ is finite. There is also a result in \cite{cao2025decreasingfiltrationsc2algebratwisted} that proves the $ C_{n+2} $-cofiniteness of $ g $-twisted $ V $-modules, for all $ n \in \mathbb{Z}_{\geq 0} $, when $ V $ and $ W $ are both assumed to be $ C_2 $-cofinite and $ g $ is assumed to have finite order. In the present work, we do not assume that $ W $ is $ C_2 $-cofinite and we allow $ g $ to be a general automorphism of $ V $ (of possibly infinite order).

This paper is organized as follows. In Section~\ref{sec:twisted modules}, we introduce a notion of a weak $ g $-twisted $ V $-module for a vertex operator algebra $ V $ together with a vertex operator algebra automorphism $ g $ of $ V $. Our notion of a weak $ g $-twisted $ V $-module will be essentially equivalent to that in \cite{bakalovtwist}. We will derive results (some similar to those in \cite{bakalovtwist} and \cite{Huang2016AssociativeAF}) that will be useful for proving the main results. In Section~\ref{sec: cofiniteness}, we introduce a notion of $ C_n $-cofiniteness for weak $ g $-twisted $ V $-modules. Given a $ C_2 $-cofinite CFT-type vertex operator algebra $ V $ and a weak $ g $-twisted $ V $-module generated by a single element, we derive certain spanning sets for the module in Proposition~\ref{prop:C2 spanning set}. Using these spanning sets, we prove Theorem~\ref{thm:Cn cofiniteness}, namely that every finitely generated weak $ g $-twisted $ V $-module is $ C_n $-cofinite for all $ n \in \mathbb{Z}_{> 0} $.
\paragraph{Notation} We write $ u(n,k) $ for the modes in $ Y(u,x) = \sum_{k = 0}^K \sum_{n \in \mathbb{C}} u(n,k) x^{-n-1} ( \log x )^k $ for clarity. The modes are also commonly written as $ u_{n,k} $, or other variations, in other literature.
\paragraph{Acknowledgments} I would like to thank Professor Yi-Zhi Huang and Professor James Lepowsky for their helpful discussions, guidance, and suggestions. This research was supported by the Department of Mathematics, Rutgers University.

\section{Weak $ g $-twisted $ V $-modules} \label{sec:twisted modules}
We use the formal calculus conventions and the notion of a vertex operator algebra $ V = (V,Y, \one, \omega) $ from \cite{LepowskyLi} (see also \cite{FLM, FHL}), and the logarithmic formal calculus introduced in \cite{HLZ2}. The base field in this paper will always be the complex numbers $ \mathbb{C} $. Recall that an automorphism of a vertex operator algebra $ V $ is an invertible linear map $ g : V \to V $ such that $ g\!\one = \one $, $ g \omega = \omega $, and $ gY(u,x)v = Y(gu,x)gv $ for all $ u,v \in V $.

Let $ V $ be a vertex operator algebra and let $ g $ be an automorphism of $ V $. Since $ g $ fixes $ \omega $, it commutes with the modes of $ Y(\omega,x) $. Since $ g $ commutes with $ L(0) $, $ g $ acts invertibly on each finite-dimensional subspace $ V_{(n)} $. Hence, there is an operator $ \mathcal{L} $ on $ V_{(n)} $ such that $ g|_{V_{(n)}} = e^{2 \pi \iu \mathcal{L}} $. The Jordan decomposition of $ \mathcal{L}  $ gives commuting operators $ \mathcal{S} $ and $ \mathcal{N} $ on $ V_{(n)} $ such that $ \mathcal{L} = \mathcal{S} + \mathcal{N} $, $ \mathcal{S} $  is semisimple and $ \mathcal{N} $ is nilpotent. Define (using the same notation) the operators $ \mathcal{L} $, $ \mathcal{S} $ and $ \mathcal{N} $ on $ V = \coprod_{n \in \mathbb{Z}} V_{(n)} $ to be the direct sums of the respective operators on each of the $ V_{(n)} $. We will choose $ \mathcal{L} $ so that the eigenvalues of $ \mathcal{S} $ have real parts in $ [0,1) $. Denote the set of $ \mathcal{S} $-eigenvalues by $ P(V) $. The vertex operator algebra has decomposition
\begin{equation}\label{eq:eigenspace decomp}
V = \coprod_{n \in \mathbb{Z}, \ \alpha \in P(V)} V^{[\alpha]}_{(n)} , \quad  \text{where } V^{[\alpha]}_{(n)} = \{ v \in V : L(0) v = n v, \ \mathcal{S} v = \alpha v \} .
\end{equation}
We also have the coarser decomposition $ V = \coprod_{ \alpha \in P(V)} V^{[\alpha]} $ given by $ V^{[\alpha]} = \coprod_{n \in \mathbb{Z}} V^{[\alpha]}_{(n)} $. We will use $ V_+ = \coprod_{n \in \mathbb{Z}_{>0}} V_{(n)} $ and $ V_+^{[\alpha]} = \coprod_{n \in \mathbb{Z}_{>0}} V^{[\alpha]}_{(n)} $ to denote the respective subspaces spanned by vectors of positive weight only. 

We will provide a definition for a weak $ g $-twisted $ V $-module. This definition will be equivalent to the notion of a $ g^{-1} $-twisted $ V $-module in Definition 4.1 of \cite{bakalovtwist} when $ g $ is locally finite (this is true when $ g $ is an automorphism of $ V $ viewed as a vertex \emph{operator} algebra). In our presentation, it will be clear that a weak $ g $-twisted $ V $-module relaxes the definition of a $ g $-twisted $ V $-module in Definition 3.1 of \cite{Gentwistmod}, with the Jacobi identity from \cite{Huang2016AssociativeAF}, by removing the grading conditions. Equivalently, we are generalizing the definition of a weak $ g $-twisted $ V $-module in Definition 3.1 of \cite{DLMtwisted} so that $ g $ can be a general automorphism of $ V $.

The definition below will be algebraic so we can perform formal calculus with the Jacobi identity immediately. The algebraic formulation of the equivariance property will use the operator $ e^{2 \pi \iu x \frac{d}{dx}} $ on $ (\End W)\{x\}[\log x] $ that is defined by
\begin{equation}
e^{2 \pi \iu x \frac{d}{dx}} \sum_{k = 0}^K \sum_{n \in \mathbb{C}} a(n,k) x^{n} ( \log x )^k = \sum_{k = 0}^K \sum_{n \in \mathbb{C}} a(n,k) e^{2 \pi \iu n} x^{n} ( \log x + 2\pi \iu )^k .
\end{equation}
This operation is the formal version of taking the variable $ x $ around the origin once.

\begin{definition} \label{def:weak g-twisted module}
A \emph{weak $ g $-twisted $ V $-module} is a vector space $ W $ equipped with a linear map
\begin{equation}
\begin{aligned}
Y_W : V &\to (\End W )\{x\} [ \log x ] \\
u &\mapsto Y_W(u, x) = \sum_{k = 0}^K \sum_{n \in \mathbb{C}} u(n,k) x^{-n-1} ( \log x )^k 
\end{aligned}
\end{equation}
satisfying the following conditions:
\begin{enumerate}
\item[1.] (\emph{The lower truncation property}) For all $ u \in V $, $ w \in W $, $ n \in \mathbb{C} $ and $ k \in \mathbb{Z}_{\geq 0} $, we have 
\begin{equation}
u(n + l,k) w = 0 , \quad \text{for all sufficiently large }  l \in \mathbb{Z} .
\end{equation}  
\item[2.] (\emph{The identity property}) $ Y_W(\one, x) = \id_W $. 
\item[3.] (\emph{The Jacobi identity}) For all $ u, v \in V $,
\begin{equation} \label{eq:Jacobi}
\begin{aligned}
&x_0^{-1} \delta \left( \frac{ x_1 - x_2}{x_0} \right)Y_W(u,x_1) Y_W(v,x_2) - x_0^{-1} \delta \left( \frac{ -x_2 + x_1}{x_0} \right)Y_W(v,x_2) Y_W(u,x_1) \\
& \qquad = x_1^{-1} \delta \left( \frac{ x_2 + x_0}{x_1} \right) Y_W \left( Y \left(   \left( \frac{x_2 + x_0}{x_1} \right)^{\mathcal{L}} u, x_0 \right) v,x_2 \right).
\end{aligned} 
\end{equation} 
\item[4.] (\emph{The equivariance property}) For all $ u \in V $,
\begin{equation}
e^{2 \pi \iu x \frac{d}{dx}} Y_W(gu,x) =  Y_W(u,x).
\end{equation}
\end{enumerate}
\end{definition}

The Jacobi identity \eqref{eq:Jacobi} is equivalent to the Borcherds identity presented in Propositions 5.2 and 5.3 of \cite{bakalovtwist}. The $ g $-twisted $ V $-modules and their variations in \cite{FLM}, \cite{DongTwisted}, \cite{DL_relative_twisted}, \cite{DLMtwisted}, \cite{Gentwistmod}, \cite{bakalovtwist} (when applied to vertex operator algebras), and \cite{Huang2016AssociativeAF} are examples of weak $ g $-twisted $ V $-modules. In \cite{Huang2016AssociativeAF}, the Jacobi identity \eqref{eq:Jacobi} was derived from the analytic ``duality property" used in \cite{Gentwistmod} and \cite{Huang2016AssociativeAF}. Note that our definition does not assume the $ L(-1) $-derivative property  
\begin{equation}
\frac{d}{dx} Y_W(u,x) = Y_W(L(-1)u,x), \quad \text{for all } u \in V ,
\end{equation}
where 
\begin{equation}
\frac{d}{dx} = \frac{\partial}{\partial x} + x^{-1} \frac{\partial}{\partial (\log x)}.
\end{equation}
Recall that we treat $ x $ and $ \log x $ as independent formal variables when we use the formal partial derivatives $ \frac{\partial}{\partial x} $ and $ \frac{\partial}{\partial (\log x)} $.

We will prove the $ L(-1) $-derivative property as a consequence later in Proposition~\ref{prop:L(-1)}.

We now show that $ \mathcal{N} $ can be used to remove logarithms from $ Y_W $ as proved in \cite{Huang2016AssociativeAF} for $ \overline{\mathbb{C}}_+ $-graded weak $ g $-twisted $ V $-modules. We briefly reprove the analogous results here to show that they hold for our relaxed notion of a weak $ g $-twisted $ V $-module without a $ g $-action on $ W $, any grading on $ W $, nor the $ L(-1) $-derivative property.

Define the operators $ e^{\pm 2 \pi \iu x \frac{\partial}{\partial x}} $ and $ e^{\pm 2 \pi \iu  \frac{\partial}{\partial (\log x)}} $ on $(\End W)\{x\}[\log x] $ by
\begin{align}
e^{\pm 2 \pi \iu x \frac{\partial}{\partial x}} \sum_{k = 0}^K \sum_{n \in \mathbb{C}} a(n,k) x^{n} (\log x)^k &= \sum_{k = 0}^K \sum_{n \in \mathbb{C}} a(n,k) e^{\pm 2 \pi \iu n} x^{n} (\log x)^k \quad \text{and} \\
\quad e^{\pm 2 \pi \iu  \frac{\partial}{\partial (\log x)}} \sum_{k = 0}^K \sum_{n \in \mathbb{C}} a(n,k) x^{n} (\log x)^k &= \sum_{k = 0}^K \sum_{n \in \mathbb{C}} a(n,k) x^{n} ( \log x \pm 2\pi \iu )^k .
\end{align}
Observe that $ e^{\pm 2 \pi \iu  \frac{\partial}{\partial (\log x)}} $ acts as the finite sum $ \sum_{j \geq 0} \frac{1 }{j!} \left(\pm 2 \pi \iu \frac{\partial}{\partial (\log x)} \right)^j $ when acting on polynomials in $ \log x $. However, since formal calculus does not involve any analytic convergence of sums, we will avoid identifying $ e^{\pm 2 \pi \iu x \frac{\partial}{\partial x}} $ with the infinite sum $ \sum_{j \geq 0} \frac{1}{j!} \left( \pm 2 \pi \iu x \frac{\partial}{\partial x} \right)^j $.

Let $ u \in V^{[\alpha]} $. Then, for sufficiently large $ N \in \mathbb{Z}_{>0} $, we have
\begin{align*}
0 &=  Y_W \left( (g - e^{2 \pi \iu \alpha})^N u,x \right)  =  \left( e^{-2 \pi \iu x \frac{\partial}{\partial x}} e^{-2 \pi \iu  \frac{\partial}{\partial (\log x)}}    - e^{2 \pi \iu \alpha} \right)^N  Y_W(u,x) \\
&= \sum_{k = 0}^K \sum_{n \in \mathbb{C}} u(n,k) \left(e^{-2 \pi \iu x \frac{\partial}{\partial x}} e^{-2 \pi \iu  \frac{\partial}{\partial (\log x)}}   - e^{2 \pi \iu \alpha} \right)^N  x^{-n-1} (\log x)^k \\
&= \sum_{k = 0}^K \sum_{n \in \mathbb{C}} u(n,k) e^{2 \pi \iu (n+1)N} \left( e^{-2 \pi \iu  \frac{\partial}{\partial (\log x)}}   - e^{2 \pi \iu (\alpha-n-1)} \right)^N  x^{-n-1} (\log x)^k .
\end{align*}
For each $ n \in \mathbb{C} $, the coefficient of $ x^{-n-1} $ gives 
\begin{align*}
0 &= \left( e^{-2 \pi \iu  \frac{\partial}{\partial (\log x)}}   - e^{2 \pi \iu (\alpha-n-1)} \right)^N \sum_{k = 0}^K u(n,k)    (\log x)^k .
\end{align*}
However, the only possible eigenvector of the translation $ e^{-2 \pi \iu  \frac{\partial}{\partial (\log x)}} $ is a constant with corresponding eigenvalue $ 1 $. So, $ \sum_{k = 0}^K u(n,k) (\log x)^k = 0 $ whenever $ \alpha - n - 1 \notin \mathbb{Z} $. Hence, for all $ u \in V^{[\alpha]} $, we have
\begin{equation} \label{eq:sum index}
Y_W(u, x) = \sum_{k = 0}^K \sum_{n \in \alpha + \mathbb{Z}} u(n,k) x^{-n-1} ( \log x )^k ,
\end{equation} 
implying
\begin{equation} \label{eq:partial relation 1}
 e^{-2 \pi \iu x \frac{\partial}{\partial x}} Y_W(u, x) =  e^{2 \pi \iu \alpha} Y_W(u, x) = Y_W(e^{2 \pi \iu \mathcal{S}} u, x).
\end{equation} 
Hence,
\begin{equation} \label{eq:partial relation 2}
\begin{aligned}
e^{-2\pi \iu \frac{\partial}{ \partial (\log x)}} Y_W(u, x) &=  e^{2 \pi \iu x \frac{\partial}{\partial x}} e^{-2 \pi \iu x \frac{\partial}{\partial x}} e^{-2\pi \iu \frac{\partial}{\partial (\log x)} } Y_W(u, x)  \\
&=   Y_W( e^{-2 \pi \iu \mathcal{S}}g u, x) = Y_W( e^{2 \pi \iu \mathcal{N}}u, x). 
\end{aligned}
\end{equation} 
The space $ V $ is spanned by generalized eigenvectors of $ g $, hence \eqref{eq:partial relation 1} and \eqref{eq:partial relation 2} are also true for general $ u \in V $. 

Recall that the algebraic logarithm power series $ \log(1 - x) = \sum_{k \geq 1} \frac{-1}{k} x^k $ provides an ``inverse" to the power series $ e^x $ in the sense that $ \log(1 - (1-e^x)) = x $.  Since $ \mathcal{N} $ and $ \frac{\partial}{\partial (\log x)} $ are both locally nilpotent operators on their respective spaces, we find that  
\begin{align*}
 \log(1 - (1-e^{2 \pi \iu \mathcal{N}})) u &= 2 \pi \iu \mathcal{N} u  \ \ \text{ and } \ \ \\
 \log(1 - (1-e^{-2\pi \iu \frac{\partial}{ \partial (\log x)}})) Y_W(u, x) &= -2\pi \iu \frac{\partial}{ \partial (\log x)} Y_W(u, x),
\end{align*}
where $ e^x $ and $ \log(1 - x) $ both denote some finite truncation of their infinite sum formulas.

Note that this works since the operators $ \mathcal{N} $ and $ \frac{\partial}{ \partial (\log x)} $ are locally nilpotent; otherwise, we would have to use a multi-valued analytic logarithm. 

\begin{proposition}
For all $ u \in V $, we have
\begin{gather} 
 - \frac{\partial}{\partial (\log x)} Y_W(u, x)  = Y_W( \mathcal{N} u, x) \label{eq:log derivative}, \\
Y_W( x^{\mathcal{N}} u, x) =  (Y_W)_0(  u, x) \label{eq: log remove},
\end{gather} 
where we use $ (Y_W)_0(u,x) = \sum_{n \in \mathbb{C}} u(n,0) x^{-n-1} $ to denote the series of $(\log x)$-free terms in $ Y_W $. 
\end{proposition} 

\begin{proof}
From \eqref{eq:partial relation 2}, we can show that
\begin{align*}
\left(1 - e^{-2\pi \iu \frac{\partial}{ \partial (\log x)}} \right)^k Y_W(u, x) = Y_W\left( \left( 1 - e^{2 \pi \iu \mathcal{N}}\right)^k u, x \right).
\end{align*}
Hence, 
\begin{align*}
 - 2 \pi \iu \frac{\partial}{\partial (\log x)} Y_W(u, x) &= \log(1 - (1-e^{-2\pi \iu \frac{\partial}{ \partial (\log x)}})) Y_W(u, x) \\
 &= Y_W\left( \log(1 - (1-e^{2 \pi \iu \mathcal{N}}))u,x \right) = Y_W(2 \pi \iu \mathcal{N} u ,x),
\end{align*}
and we have \eqref{eq:log derivative}. Then formally exponentiating \eqref{eq:log derivative} gives
\begin{align*}
 Y_W( x^{\mathcal{N}} u, x) &= Y_W( e^{\log x \mathcal{N}} u, x) = \sum_{j \geq 0} \frac{1}{j!} (\log x)^j Y_W( \mathcal{N}^j u, x) \\
 &= \sum_{j \geq 0} \frac{(-1)^j}{j!} (\log x)^j \left(\frac{\partial}{\partial (\log x)}\right)^j Y_W( u, x) = Y_W( u, x)|_{\log x = 0} =  \sum_{n \in \mathbb{C}} u(n,0) x^{-n-1} . \qedhere
\end{align*} 
\end{proof}

We have just seen how $ \mathcal{N} $ interacts with $ Y_W $. Now we will see how $ \mathcal{N} $ interacts with $ Y $. Observe that we have the Leibniz rule
\begin{align*}
\left( g - e^{2\pi \iu (\alpha+ \beta)} \right) \Big( Y(u,x)v \Big) = \Big( Y \left( (g - e^{2\pi \iu \alpha} )u, x \right) g \Big) v + Y(u,x) \Big( e^{2\pi \iu \alpha} ( g -e^{2\pi \iu \beta}) v \Big).
\end{align*}
Hence,
\begin{equation} \label{eq:gen eigen}
\left( g - e^{2\pi \iu (\alpha+ \beta)} \right)^N \Big( Y(u,x)v \Big) = \sum_{j = 0}^N \binom{N}{j} Y \left( (g - e^{2\pi \iu \alpha} )^{N-j} u, x \right) g^{N-j}  e^{2\pi \iu \alpha j} ( g -e^{2\pi \iu \beta})^j v  .
\end{equation}
If $ u $ and $ v $ are generalized eigenvectors of $ g $ with eigenvalues $ e^{2 \pi \iu \alpha} $ and $  e^{2 \pi \iu \beta} $, respectively, we can choose $ N $ sufficiently large so that the right-hand side of \eqref{eq:gen eigen} is zero. This shows that 
\begin{equation}
e^{2 \pi \iu \mathcal{S}} Y(u,x) e^{-2 \pi \iu \mathcal{S}} = Y( e^{2 \pi \iu \mathcal{S}}u,x).
\end{equation}
Hence,  
\begin{equation}
e^{2 \pi \iu \mathcal{N}} Y(u,x) e^{-2 \pi \iu \mathcal{N}} = Y( e^{2 \pi \iu \mathcal{N}}u,x).
\end{equation}
Define $ L_{\mathcal{N}} $ and $ R_{\mathcal{N}} $ to be operators on $ \End V $ that compose operators on $ V $ with $ \mathcal{N} $ on the left and right, respectively. Since $ L_{\mathcal{N}} $ and $ R_{\mathcal{N}} $ commute, we have 
\begin{align*}
(L_{\mathcal{N}} - R_{\mathcal{N}})^N f v  = \sum_{j =0}^N \binom{N}{j}(-1)^j L_{\mathcal{N}}^{N-j}R_{\mathcal{N}}^j f v = \sum_{j =0}^N \binom{N}{j}(-1)^j \mathcal{N}^{N-j} f \mathcal{N}^j v = 0
\end{align*} 
for sufficiently large $ N $, dependent on $ f \in \End V  $ and $ v \in V $. Since the sums in the exponentials in each coefficient of $ x^{-n-1} $ in 
\begin{align*}
e^{2\pi\iu (L_\mathcal{N} - R_\mathcal{N})} Y(u,x) v = e^{2 \pi \iu \mathcal{N}} Y(u,x) e^{-2 \pi \iu \mathcal{N}} v  =  Y( e^{2 \pi \iu \mathcal{N}} u, x )v
\end{align*}
are finite, we have
\begin{align*}
2\pi\iu (L_\mathcal{N} - R_\mathcal{N})Y(u,x) v  &= \log(1 - ( 1- e^{2\pi\iu (L_\mathcal{N} - R_\mathcal{N})})) Y(u,x) v \\
&=  Y( \log(1 - (1 - e^{2 \pi \iu \mathcal{N}}))u, x )v = Y( 2 \pi \iu \mathcal{N} u, x )v.
\end{align*}
Thus, we have
\begin{equation}\label{eq:N conj}
[ \mathcal{N}, Y(u,x) ] = Y(\mathcal{N} u,x),
\end{equation}
which can also be interpreted as saying that $ \mathcal{N} $ is a derivation of the vertex operator algebra $ V $. Formally exponentiating \eqref{eq:N conj} gives
\begin{equation} 
\label{eq: x^N conj }
x_0^{\mathcal{N}} Y(u,x)x_0^{-\mathcal{N}} = Y(x_0^{\mathcal{N}} u,x) .
\end{equation}

From the Jacobi identity \eqref{eq:Jacobi}, the ability to remove logarithms \eqref{eq: log remove} and the conjugation property \eqref{eq: x^N conj }, we have
\begin{align*}
&x_0^{-1} \delta \left( \frac{ x_1 - x_2}{x_0} \right)(Y_W)_0(u,x_1) (Y_W)_0(v,x_2) - x_0^{-1} \delta \left( \frac{ -x_2 + x_1}{x_0} \right)(Y_W)_0(v,x_2) (Y_W)_0(u,x_1) \\
& \qquad = x_1^{-1} \delta \left( \frac{ x_2 + x_0}{x_1} \right) (Y_W)_0 \left( Y \left( \left( \frac{x_2}{x_1}\right)^{\mathcal{S}} \left( 1 + \frac{ x_0}{x_2} \right)^{\mathcal{L}} u, x_0 \right) v,x_2 \right).
\end{align*} 
When $ u $ is an eigenvector for $ \mathcal{S} $ with eigenvalue $ \alpha $ (i.e., a generalized eigenvector for $ g $), we have
\begin{equation}\label{eq:jacobi Y_0}
\begin{aligned}
&x_0^{-1} \delta \left( \frac{ x_1 - x_2}{x_0} \right)(Y_W)_0(u,x_1) (Y_W)_0(v,x_2) - x_0^{-1} \delta \left( \frac{ -x_2 + x_1}{x_0} \right)(Y_W)_0(v,x_2) (Y_W)_0(u,x_1) \\
& \qquad = x_1^{-1} \delta \left( \frac{ x_2 + x_0}{x_1} \right) \left( \frac{x_2}{x_1}\right)^{\alpha} (Y_W)_0 \left( Y \left( \left( 1 + \frac{ x_0}{x_2} \right)^{\mathcal{L}} u, x_0 \right) v,x_2 \right)
\end{aligned} 
\end{equation}
(cf. equation (2.29) of \cite{Huang2016AssociativeAF}). This form of the Jacobi identity will be better suited for residue calculus.

We will now prove the $ L(-1) $-derivative property. We apply $ \res_{x_0}\res_{x_1} x_0^{-2} x_1^{\alpha}{x_2}^{-\alpha} $ to the left-hand side of \eqref{eq:jacobi Y_0} after specializing $ v = \one $ to find
\begin{align*}
&\res_{x_0}\res_{x_1}x_0^{-2} x_1^{\alpha}{x_2}^{-\alpha} \left( x_0^{-1} \delta \left( \frac{ x_1 - x_2}{x_0} \right)(Y_W)_0(u,x_1) (Y_W)_0(\one,x_2) \right. \\
& \quad - \left. x_0^{-1} \delta \left( \frac{ -x_2 + x_1}{x_0} \right)(Y_W)_0(\one,x_2) (Y_W)_0(u,x_1) \right) \\
&= \res_{x_1} x_1^{\alpha}{x_2}^{-\alpha} \left( ( x_1 - x_2)^{-2} - ( -x_2 + x_1)^{-2} \right) (Y_W)_0(u,x_1) \\
&= \res_{x_1} x_1^{\alpha}{x_2}^{-\alpha} \frac{\partial}{\partial x_1}  \left( - x_1^{-1} \delta \left( \frac{ x_2}{x_1} \right)\right)  (Y_W)_0(u,x_1) \\
&= \res_{x_1} {x_2}^{-\alpha}  x_1^{-1} \delta \left( \frac{ x_2}{x_1} \right)  \frac{\partial}{\partial x_1} \left( x_1^{\alpha} (Y_W)_0(u,x_1) \right) \\
&= \res_{x_1} {x_2}^{-\alpha}  x_1^{-1} \delta \left( \frac{ x_2}{x_1} \right)  \frac{\partial}{\partial x_2} \left( {x_2}^{\alpha} (Y_W)_0(u,x_2) \right) \\
&= {x_2}^{-\alpha}  \frac{\partial}{\partial x_2} \left( {x_2}^{\alpha} (Y_W)_0(u,x_2) \right) = \frac{\partial}{\partial x_2} (Y_W)_0(u,x_2) + \alpha x_2^{-1} (Y_W)_0(u,x_2) .
\end{align*}
Doing the same to the right-hand side gives,
\begin{align*}
&\res_{x_0}\res_{x_1}x_0^{-2} x_1^{\alpha}{x_2}^{-\alpha} \left( x_1^{-1} \delta \left( \frac{ x_2 + x_0}{x_1} \right) \left( \frac{x_2}{x_1}\right)^{\alpha} (Y_W)_0 \left( Y \left( \left( 1 + \frac{ x_0}{x_2} \right)^{\mathcal{L}} u, x_0 \right) \one ,x_2 \right) \right) \\
&=\res_{x_0} x_0^{-2} (Y_W)_0 \left( Y \left( \left( 1 + \frac{ x_0}{x_2} \right)^{\mathcal{L}} u, x_0 \right) \one ,x_2 \right)  \\
&=\res_{x_0} x_0^{-2} (Y_W)_0 \left( \sum_{n \in \mathbb{Z}}  \left( \sum_{j \geq 0 } \binom{ \mathcal{L}}{j} x_0^j x_2^{-j}  u \right)(n) x_0^{-n-1} \one,x_2 \right)  \\
&= \sum_{j \geq 0 } x_2^{-j} (Y_W)_0 \left( \left( \binom{ \mathcal{L}}{j}  u \right)(-2 + j) \one, x_2 \right)  =  (Y_W)_0 \left(u(-2) \one, x_2 \right) + x_2^{-1} (Y_W)_0 \left( \left( \mathcal{L} u \right)(-1) \one, x_2 \right) \\
&=(Y_W)_0 \left( L(-1)u, x_2 \right) + x_2^{-1} (Y_W)_0 \left( \alpha u + \mathcal{N} u , x_2 \right) .
\end{align*}
Rearranging gives
\begin{equation}\label{eq:Y_0 derivative}
(Y_W)_0 \left( L(-1)u, x \right) =  \frac{\partial}{\partial x} (Y_W)_0(u,x) - x^{-1} (Y_W)_0 \left( \mathcal{N} u , x \right).
\end{equation}
Since $ \omega $ is fixed by $ g $, we have $ \mathcal{N} \omega = 0 $, hence \eqref{eq:N conj} gives $ [L(-1), \mathcal{N} ] = 0 $. So, from \eqref{eq:Y_0 derivative}, we obtain
\begin{align*}
(Y_W)_0 \left( e^{-\log x \mathcal{N}} L(-1)u, x \right) &= (Y_W)_0 \left( L(-1) e^{-\log x \mathcal{N}} u, x \right) \\
&=  \frac{\partial}{\partial x} (Y_W)_0(e^{-\log x \mathcal{N}}u,x) - x^{-1} (Y_W)_0 \left( \mathcal{N}e^{-\log x \mathcal{N}} u , x \right) \\
&= \frac{\partial}{\partial x} (Y_W)_0(e^{-\log x \mathcal{N}}u,x) + x^{-1} \frac{\partial}{\partial (\log x )} (Y_W)_0 \left( e^{-\log x \mathcal{N}} u , x \right)  \\
&= \frac{d}{d x} (Y_W)_0(e^{-\log x \mathcal{N}}u,x).
\end{align*}
Thus, by \eqref{eq: log remove}, we obtain the following result (cf. equation (4.3) of \cite{bakalovtwist} and Definition 3.1 of \cite{Gentwistmod}).
\begin{proposition} \label{prop:L(-1)}
Let $ W $ be a weak $ g $-twisted $ V $-module. Then $ W $ satisfies the \emph{$ L(-1) $-derivative property}
\begin{equation}
Y_W( L(-1)u,x) = \frac{d}{d x} Y_W(u,x), \quad \text{for all } u \in V. \qedhere
\end{equation}
\end{proposition} 

Furthermore, since $ \omega $ is fixed by $ g $, the Jacobi identity \eqref{eq:Jacobi} lifts the Virasoro algebra action on $ V $ to a Virasoro algebra action on $ W $ with the same central charge as $ V $. Any of the usual relations involving the Virasoro modes can utilize the $ L(-1) $-derivative property. For example,
\begin{align*}
[L(0), Y_W(u,x)] = Y_W(L(0)u,x) + x Y_W(L(-1)u,x) = Y_W(L(0)u,x) + x \frac{d}{dx} Y_W(u,x) .
\end{align*}

\begin{remark}
In Definition~\ref{def:weak g-twisted module}, replacing the codomain of $ Y_W $ with $ ( \End W) [ \log x ]\{x\} $ would give an a priori weaker definition. However, \eqref{eq:log derivative} can be derived in the same manner, which, due to the local nilpotency of $ \mathcal{N} $, shows that there is a global upper bound for the powers of $ \log x $ for all powers of $ x $. Hence, the image of $ Y_W $ is contained in the original smaller codomain $ ( \End W)\{x\}[ \log x ] $.
\end{remark}

\begin{remark}
The above exposition for the $ L(-1) $-derivative property was chosen to additionally derive the important formulas \eqref{eq:jacobi Y_0}, \eqref{eq:Y_0 derivative} and the lemma below for the next section. We could have also proved the $ L(-1) $-derivative property directly from \eqref{eq:Jacobi} after proving \eqref{eq:sum index} and \eqref{eq:log derivative}. If this approach is taken, one must carefully distinguish between $ \frac{\partial}{\partial x_1}  $ and $ \frac{d}{d x_1} = \frac{\partial}{\partial x_1} + x_1^{-1} \frac{\partial}{\partial (\log x_1)} $ when performing $ \res_{x_1} $.
\end{remark}

The $ (\log x) $-free variant of the $ L(-1) $-derivative property \eqref{eq:Y_0 derivative} provides a lemma that will be useful in the following section.
\begin{lemma} \label{lem:lower modes}
Let $ u \in V $ and $ n \in \mathbb{C} \backslash \mathbb{Z}_{\geq 0} $. Then for all $ j \in \mathbb{Z}_{>0} $, 
\begin{equation}
u(n-j,0) = \sum_{i \in I} v_i(n,0) ,
\end{equation}
for some elements $ v_i \in V $ with a finite index set $ I $.
\end{lemma}

\begin{proof}
Equation \eqref{eq:Y_0 derivative} gives
\begin{align*}
(L(-1)u)(n,0)  = (-n)u(n-1,0) - (\mathcal{N}u)(n-1,0).
\end{align*}
We repeatedly use this relation to obtain
\begin{align*}
u(n-1,0) &= \sum_{k = 1}^K \left( \frac{-1}{n}\right)^k (L(-1)\mathcal{N}^{k-1} u)(n,0) + \left(\frac{-1}{n}\right)^K (\mathcal{N}^{K} u)(n-1,0) \\
&= \sum_{k = 1}^\infty \left( \frac{-1}{n}\right)^k (L(-1)\mathcal{N}^{k-1} u)(n,0) ,
\end{align*}
where $ K $ is chosen sufficiently large so that $ \mathcal{N}^{K} u = 0 $. Since $ n \notin \mathbb{Z}_{\geq 0} $, we can apply this repeatedly to obtain
\begin{align*}
u(n-j,0) = \sum_{k_1, \dots, k_j = 1}^\infty \left( \frac{-1}{n-j+1}\right)^{k_j} \cdots \left( \frac{-1}{n}\right)^{k_1} (L(-1)^j \mathcal{N}^{k_1 + \dots + k_j -j} u)(n,0).
\end{align*}
This sum is finite from the local nilpotency of $ \mathcal{N} $.
\end{proof}

\section{Cofiniteness properties} \label{sec: cofiniteness}

We will introduce a notion of $ C_n $-cofiniteness for weak $ g $-twisted $ V $-modules. At the end of this section, we will show that all finitely generated weak $ g $-twisted $ V $-modules for a $ C_2 $-cofinite CFT-type vertex operator algebra $ V $ are $ C_n $-cofinite for all $ n \in \mathbb{Z}_{>0} $.

\begin{definition}
Let $ W $ be a weak $ g $-twisted $ V $-module. For each $ n \in \mathbb{Z}_{>0} $, we define $ C_n(W) $ to be the subspace of $ W $ spanned by the elements
\begin{equation}
u(\alpha -  n ,0) w , \qquad \text{for all } u \in V^{[\alpha]}, \  \alpha \in P(V), \ w \in W .
\end{equation}
When $ n = 1 $, we impose the extra condition that $ u \in V_+ = \coprod_{m \in \mathbb{Z}_{> 0}} V_{(m)} $.  We say that $ W $ is \emph{$ C_n $-cofinite} if $ \dim W / C_n(W) < \infty $.
\end{definition}

\begin{remark}
For $ C_1 $-cofiniteness, we impose $ u \in V_+ $ to avoid the case $ u = \one $. Since $ \one(-1) w = \id_W w = w $, this case would trivially make $ C_1(W) = W $. Our definition is analogous to the definition of $ C_1 $-cofiniteness in \cite{HuangDEs} that provides convergence of correlation functions formed by a product of intertwining operators. Other definitions for $ C_1 $-cofiniteness can be made, for example in \cite{Lifinite}.
\end{remark}

From \eqref{eq: log remove}, we have
\begin{equation} \label{eq:Y to Y_0}
\begin{aligned}
Y_W(u,x) &= (Y_W)_0(x^{-\mathcal{N}} u,x) = \sum_{k \geq 0} \frac{(-1)^k}{k!} (Y_W)_0(\mathcal{N}^k u,x) (\log x)^k \\
&= \sum_{k = 0}^K \sum_{n \in \mathbb{C}}  \frac{(-1)^k}{k!} (\mathcal{N}^k u)(n,0) x^{-n-1} (\log x)^k.
\end{aligned}
\end{equation}  
So, we see that $ C_n(W) $ is actually equal to the (a priori larger) subspace spanned by the elements
\begin{equation}
u(\alpha -  n ,k) w , \qquad \text{for all } u \in V^{[\alpha]}, \  \alpha \in P(V),\  k \in \mathbb{Z}_{\geq 0},\ w \in W.
\end{equation}
When $ n = 1 $, we still impose the extra condition that $ u \in V_+ $. Using \eqref{eq:sum index}, \eqref{eq: log remove} and the linearity of $ Y_W $, we see that we can also equivalently define
\begin{equation}
\begin{aligned}
C_1(W) &= \Span \{ \res_x x^{-1} Y_W ( x^\mathcal{L} u , x) w : \text{for all } u \in V_+ \text{ and } w \in W \} \quad \text{and} \\
C_n(W) &= \Span \{ \res_x x^{-n} Y_W ( x^\mathcal{L} u , x) w : \text{for all } u \in V \text{ and } w \in W \} \quad \text{ when } n \geq 2 .
\end{aligned}
\end{equation}
Note that when $ g = \id_V $, this definition agrees with the usual definition of $ C_n $-cofiniteness for untwisted $ V $-modules, specifically when $ n = 2 $ and $ V $ is itself viewed as an untwisted $ V $-module.

\begin{definition}
We say a vertex operator algebra $ V $ is of \emph{CFT type} (also known as \emph{positive energy}) if $ V_{(n)} = 0 $ when $ n < 0 $, and $ V_{(0)} = \mathbb{C}\one $.
\end{definition}

The CFT type condition is a natural assumption satisfied by common examples: affine vertex operator algebras, lattice vertex operator algebras, Virasoro vertex operator algebras, etc. 
 
\begin{proposition}
Let $ n \in \mathbb{Z}_{\geq 3} $ and let $ W $ be a $ C_n $-cofinite $ g $-twisted $ V $-module. Then $ W $ is $ C_{n-1} $-cofinite. Furthermore, if $ V $ is of CFT type and $ W $ is $ C_2 $-cofinite, then $ W $ is $ C_1 $-cofinite.
\end{proposition}

\begin{proof}
We know that $ C_n (W) $ is spanned by 
\begin{align*}
u(\alpha -  n ,0) w , \qquad \text{for all } u \in V^{[\alpha]}, \  \alpha \in P(V), \ w \in W .
\end{align*} 
Since $ \alpha -  (n-1) \notin \mathbb{Z}_{\geq 0} $, by Lemma~\ref{lem:lower modes}, we can re-express $ u(\alpha -  n ,0)w $ as a finite sum $ \sum_{i \in I} v_i(\alpha - (n-1),0)w $ for some $ v_i \in V $. Hence $ C_n(W) \subseteq C_{n-1}(W) $. Thus, $ \dim W/ C_{n-1}(W) \leq \dim W/ C_{n}(W) < \infty $. 

In the case that $ n = 2 $ and $ V $ is of CFT type, the proof of Lemma~\ref{lem:lower modes} shows that the $ v_i $ are in the image of $ L(-1) $, hence are indeed in $ V_+ $.
\end{proof}

In what follows, we fix a weak $ g $-twisted $ V $-module $ W $. From here on, we will usually omit the subscript $ W $ from $ Y_W $ and $ (Y_W)_0 $ for brevity. Whenever omitted, it will be clear from context whether $ Y $ is for the algebra $ V $ or the module $ W $. We will also write $ u(n) $ to denote $ u(n,0) $.

We associate to $ W $, the algebra $ E $ defined to be the subalgebra of $ \End W $ generated by the coefficients of $ Y(u,x) \in (\End W) \{x\} [ \log x ] $, for all $ u \in V $. By \eqref{eq:eigenspace decomp} and \eqref{eq:Y to Y_0}, it is sufficient to think of $ E $ as the subalgebra of $ \End W $ generated by the coefficients of $ Y_0(u,x) \in (\End W) \{x\} $ for all $ u \in V^{[\alpha]} $, $ \alpha \in P(V) $. Define the filtration 
\begin{equation}
E^{(0)} \subseteq E^{(1)} \subseteq \cdots \subseteq E,
\end{equation}
where $ E^{(s)} $ is spanned by elements of the form 
\begin{equation}
u_1(\alpha_1 - n_1) \cdots u_k(\alpha_k - n_k) 
\end{equation}
with $ k \in \mathbb{Z}_{>0} $,  $ n_j \in \mathbb{Z} $, and $ u_j \in V^{[\alpha_j]} $ satisfying $ \sum_{j=1}^k \wt {u_j} \leq s $. This filtration is compatible with the algebra structure on $ E $ in the sense that $ E^{(s)} E^{(t)} \subseteq E^{(s+t)} $. For each $ E^{(s)} $, we define the filtration 
\begin{equation}
 \cdots \subseteq E^{(s,-1)}  \subseteq E^{(s,0)} \subseteq E^{(s,1)} \subseteq \cdots \subseteq E^{(s)} ,
\end{equation}
where $ E^{(s,t)} $ is spanned by elements of the form 
\begin{equation}
u_1(\alpha_1 - n_1) \cdots u_k(\alpha_k - n_k) 
\end{equation}
with $ k \in \mathbb{Z}_{>0} $,  $ n_j \in \mathbb{Z} $, $ u_j \in V^{[\alpha_j]} $ satisfying $ \sum_{j=1}^k \wt {u_j} \leq s $, and there is some $ i \in\{ 1, \dots, k \}$ such that $ n_i \geq -t $. 

Recall that a sequence $ (f_i)_{i = 0}^\infty $ of endomorphisms of $ W $ is \emph{summable} if for every $ w \in W $, $ f_i w = 0 $ for all but finitely many $ i \geq 0 $. In this case, we write $ \sum_{ i \geq 0} f_i $ to denote the element in $ \End W $ defined by $ \left(\sum_{ i \geq 0} f_i \right) w = \sum_{ i \geq 0} f_i w $. Given a subspace $ U \subseteq \End W $, we will write $ \sum_{ i \geq 0} f_i \in U $ as shorthand to denote that $ f_i \in U $ for each $ i \geq 0 $.

\begin{remark}
Notice that the filtrations for $ E $ come from the $ L(0) $-grading on $ V $ only. This is why we do not need grading conditions on $ W $. 
\end{remark}

In the following, we will assume that $ u $ and $ v $ are weight-homogeneous eigenvectors of $ \mathcal{S} $ with eigenvalue $ \alpha $ and $ \beta $, respectively. Since $ L(0) $ and $ \mathcal{S} $  act semisimply and commute, this assumption is sufficient to describe all general $ u , v \in V $. Denote the $ \mathcal{S} $-eigenvalue of $ u(n) v $ by $ s(\alpha,\beta) $, that is
\begin{align*}
s(\alpha,\beta) = \begin{cases}
\alpha + \beta &\text{if } \Re(\alpha+ \beta) < 1, \\
\alpha + \beta - 1 &\text{if } \Re(\alpha + \beta) \geq 1. \\
\end{cases}
\end{align*}

\begin{lemma}
\label{lem:swap}
For all $ m, n \in \mathbb{Z} $, we have
\begin{equation}
[u( \alpha + m), v(\beta + n) ] = \sum_{j \geq 0}\left(  \left( \binom{m +\mathcal{L}}{j} u \right) (j) v \right)(\alpha + \beta + m+n - j).
\end{equation} 
Furthermore, 
\begin{equation}
[u(\alpha + m), v(\beta + n) ] \in E^{(\wt u + \wt v - 1)} .
\end{equation} 
\end{lemma}

\begin{proof}
Let $ m,n \in \mathbb{Z} $. From the Jacobi identity \eqref{eq:jacobi Y_0}, we have
\begin{align*}
&[u(\alpha + m), v(\beta + n) ] \\
& = \res_{x_0} \res_{x_1} \res_{x_2} x_1^{\alpha + m} x_2^{\beta + n} \left( x_0^{-1} \delta \left( \frac{ x_1 - x_2}{x_0} \right)Y_0(u,x_1) Y_0(v,x_2) \right. \\
&\qquad \left.- x_0^{-1} \delta \left( \frac{ -x_2 + x_1}{x_0} \right)Y_0(v,x_2) Y_0(u,x_1) \right) \\
&= \res_{x_0} \res_{x_1} \res_{x_2} x_1^{\alpha + m} x_2^{\beta + n} \left(  x_1^{-1} \delta \left( \frac{ x_2 + x_0}{x_1} \right) \left( \frac{x_2}{x_1}\right)^{\alpha} Y_0 \left( Y \left( \left( 1 + \frac{ x_0}{x_2} \right)^{\mathcal{L}} u, x_0 \right) v,x_2 \right) \right) \\
&= \res_{x_0} \res_{x_2}  x_2^{\alpha+\beta + n} (x_2 + x_0)^{m} Y_0 \left( Y \left( \left( 1 + \frac{ x_0}{x_2} \right)^{\mathcal{L}} u, x_0 \right) v,x_2 \right) \\
&= \res_{x_0} \res_{x_2}  x_2^{\alpha+\beta + m + n}  Y_0 \left( Y \left( \left( 1 + \frac{ x_0}{x_2} \right)^{m + \mathcal{L}} u, x_0 \right) v,x_2 \right) \\
&= \res_{x_0} \res_{x_2}  x_2^{\alpha+\beta + m + n}  \sum_{p \in \mathbb{Z}} \sum_{q \in s(\alpha,\beta) + \mathbb{Z} }  \left( \left( \sum_{j \geq 0} \binom{ m + \mathcal{L}}{j} x_0^j x_2^{-j}  u \right) (p) v \right)(q) x_0^{-p-1} x_2^{-q-1} \\
&= \sum_{j \geq 0}\left(  \left( \binom{m +\mathcal{L}}{ j} u \right) (j) v \right)(\alpha + \beta + m+n - j) . 
\end{align*} 
 We observe that, for all $ j \geq 0 $, 
\begin{align*}
\wt \left( \binom{ m + \mathcal{L} }{j} u \right) (j) v &= \wt u - j - 1 + \wt v \leq  \wt u  + \wt v - 1. 
\end{align*} 
Hence, the sum is finite and an element in $ E^{(\wt u + \wt v - 1)} $.
\end{proof}

\begin{lemma}
\label{lem:C2 elements}
For all $ m, n \in \mathbb{Z} $, we have
\begin{equation}
\begin{aligned}
(u(m) v)(s(\alpha,\beta) +n ) &= -\sum_{j \geq 1} \left( \left( \binom{ \mathcal{L} }{j}u \right)(m +j) v \right)(s(\alpha,\beta) +n - j) \\
& \quad +  \sum_{j \geq 0} (-1)^j \binom{m}{j} u(\alpha + m -j)v( s(\alpha,\beta) - \alpha + n + j ) \\
&  \quad + \sum_{j \geq 0} (-1)^{m-j+1} \binom{m}{j} v( s(\alpha,\beta) - \alpha + m+ n-j) u(\alpha + j). 
\end{aligned} 
\end{equation}
Furthermore, 
\begin{equation}
(u(m) v)(s(\alpha,\beta) +n ) \in E^{(\wt u(m) v - 1)}, \qquad \text{when } m \leq -2 ,
\end{equation}
and in particular, all modes of $ Y_0(u,x) $ are in $ E^{(\wt u - 1)} $ whenever $ u $ is a weight-homogeneous element in $ C_2(V) $.
\end{lemma}

\begin{proof}
Let $ m,n \in \mathbb{Z} $. From the Jacobi identity \eqref{eq:jacobi Y_0}, we have
\begin{align*}
& \res_{x_0} \res_{x_1} \res_{x_2} x_0^m x_1^\alpha x_2^{s(\alpha,\beta) - \alpha + n} x_1^{-1} \delta \left( \frac{ x_2 + x_0}{x_1} \right) \left( \frac{x_2}{x_1}\right)^{\alpha} Y_0 \left( Y \left( \left( 1 + \frac{ x_0}{x_2} \right)^{\mathcal{L}} u, x_0 \right) v,x_2 \right) \\
&=  \res_{x_0} \res_{x_2} x_0^m  x_2^{s(\alpha,\beta) + n} Y_0 \left( Y \left( \left( 1 + \frac{ x_0}{x_2} \right)^{\mathcal{L}} u, x_0 \right) v,x_2 \right) \\
&=  \res_{x_0} \res_{x_2} x_0^m  x_2^{s(\alpha,\beta)+ n} \sum_{p \in \mathbb{Z}} \sum_{q \in s(\alpha,\beta) + \mathbb{Z}} \left( \left( \sum_{j \geq 0} \binom{\mathcal{L}}{j} x_0^j   x_2^{- j} u \right) (p) v \right)(q) x_0^{-p-1} x_2^{-q-1} \\
&= \sum_{j \geq 0} \left( \left( \binom{\mathcal{L}}{j}u \right)(m +j) v \right)(s(\alpha,\beta) +n - j),  
\end{align*}
which is equal to
\begin{align*}
& \res_{x_0} \res_{x_1} \res_{x_2} x_0^m x_1^\alpha x_2^{s(\alpha,\beta) - \alpha + n} \left(x_0^{-1} \delta \left( \frac{ x_1 - x_2}{x_0} \right)Y_0(u,x_1) Y_0(v,x_2) \right. \\
& \qquad \left. - x_0^{-1} \delta \left( \frac{ -x_2 + x_1}{x_0} \right)Y_0(v,x_2) Y_0(u,x_1)  \right) \\
&= \res_{x_1} \res_{x_2} x_1^\alpha x_2^{s(\alpha,\beta) - \alpha + n} \Big( (x_1 - x_2)^m Y_0(u,x_1) Y_0(v,x_2) - (- x_2 + x_1)^m Y_0(v,x_2) Y_0(u,x_1)  \Big) \\
&= \res_{x_1} \res_{x_2} x_1^\alpha x_2^{s(\alpha,\beta) - \alpha + n} \Bigg( \sum_{j \geq 0} (-1)^j \binom{m}{j} x_1^{m-j} x_2^j \sum_{p \in \alpha + \mathbb{Z}} u(p) x_1^{-p-1} \sum_{q \in \beta + \mathbb{Z}} v(q) x_2^{-q-1} \\
&\qquad \qquad + \sum_{j \geq 0} (-1)^{m-j+1} \binom{m}{j} x_2^{m-j} x_1^j \sum_{q \in \beta + \mathbb{Z}} v(q) x_2^{-q-1}   \sum_{p \in \alpha + \mathbb{Z}} u(p) x_1^{-p-1} \Bigg) \\
&= \sum_{j \geq 0} (-1)^j \binom{m}{j} u(\alpha + m -j)v( s(\alpha,\beta) - \alpha + n + j ) \\
& \qquad + \sum_{j \geq 0} (-1)^{m-j+1} \binom{m}{j} v( s(\alpha,\beta) - \alpha + m+ n-j) u(\alpha + j).
\end{align*}
Rearranging for the term $ (u(m)v)(s(\alpha, \beta) +n)  $ gives the desired relation. Observe that, for all $ j \geq 1 $,
\begin{align*}
\wt \left( \binom{\mathcal{L}}{j} u \right)(m + j) v = \wt u - m - j - 1 + \wt v < \wt u - m - 1 + \wt v = \wt (u(m) v ), 
\end{align*} 
and
\begin{align*}
\wt u + \wt v < \wt u - m - 1 + \wt v = \wt (u(m) v )
\end{align*}
when $ m \leq - 2 $. The summands vanish for sufficiently large $ j \geq  0 $ when acting on a module element. Hence, each series is summable and each summand is in $ E^{(\wt u(m) v - 1)} $.
\end{proof}

\begin{remark}
The last inequality in the proof of the previous lemma shows why $ C_2 $-cofiniteness of $ V $ is important.
\end{remark}

We now introduce a notion of normal ordering for the $(\log x)$-free parts of twisted vertex operators. We can extend these definitions to the twisted vertex operators containing logarithms using \eqref{eq: log remove}, but we will not need this here. We define
\begin{equation}
Y_0^+ (u, x) =\sum_{n \in \alpha + \mathbb{Z}_{<0}} u(n) x^{-n-1} \quad \text{and} \quad Y_0^- (u, x) = \sum_{n \in \alpha + \mathbb{Z}_{\geq 0}} u(n) x^{-n-1} .
\end{equation}
The notations for the ``regular" and ``singular" parts of $ Y_0(u,x) $ are used here because $ x^\alpha Y_0^+ (u, x) $ and $ x^\alpha Y_0^- (u, x) $ contain the non-negative and negative integral powers of $ x $, respectively. Let $ a(x) \in (\End W)\{x\} $ be such that $ a(x) w \in W\{x\} $ is lower truncated for all $ w \in W $. Define the \emph{normal ordering} of $ Y_0(u,x_1) $ and $ a(x_2) $ to be
\begin{equation}\
\coln Y_0(u,x_1) a(x_2) \coln = Y_0^+(u,x_1) a(x_2) + a(x_2) Y_0^-(u,x_1).
\end{equation}
The specialization $ x_1 = x_2 = x $ is allowed, and provides the well-defined normal ordering
\begin{equation}\label{eq:norm order def}
\coln Y_0(u,x) a(x)\coln = Y_0^+(u,x) a(x)+ a(x) Y_0^-(u,x).
\end{equation}
Since $ \mathcal{S} $ acts semisimply on $ V $, normal ordering can be linearly extended to general $ u \in V $ by defining
\begin{equation}
Y_0^+ (u, x) =\sum_{n \in \mathbb{C}, \ \Re(n) < 0} u(n) x^{-n-1} \quad \text{and} \quad Y_0^- (u, x) = \sum_{n \in \mathbb{C}, \ \Re(n) \geq 0} u(n) x^{-n-1} .
\end{equation}

\begin{lemma}
\label{lem:norm order}
We have
\begin{equation}\label{eq:take u(-1) out Y}
\begin{aligned}
\coln Y_0(u,x) Y_0(v,x) \coln &= \sum_{j \geq 0} x^{-j} Y_0 \left( \left( \binom{\mathcal{L}}{j} u \right) (-1+j) v, x \right) .
\end{aligned}
\end{equation}
\end{lemma}

\begin{proof}
Applying $ \res_{x_0} \res_{x_1} x_0^{-1} x_1^{\alpha}x_2^{-\alpha}  $ to the left-hand side of the Jacobi identity \eqref{eq:jacobi Y_0} gives
\begin{align*}
& \res_{x_0} \res_{x_1} x_0^{-1} x_1^{\alpha}x_2^{-\alpha} \left( x_0^{-1} \delta \left( \frac{ x_1 - x_2}{x_0} \right)Y_0(u,x_1) Y_0(v,x_2) \right. \\
& \qquad \left. - x_0^{-1} \delta \left( \frac{ -x_2 + x_1}{x_0} \right)Y_0(v,x_2) Y_0(u,x_1) \right) \\
&= \res_{x_1} x_1^{\alpha} x_2^{-\alpha} \left( (x_1 - x_2)^{-1} Y_0(u,x_1) Y_0(v,x_2) - (-x_2 + x_1)^{-1} Y_0(v,x_2) Y_0(u,x_1) \right) \\
&= \res_{x_1} x_1^{\alpha} x_2^{-\alpha} \left( \sum_{j \geq 0} x_1^{-1 - j} x_2^j \sum_{n \in \alpha + \mathbb{Z} } u(n) x_1^{-n-1} Y_0(v,x_2) + \sum_{j \geq 0} x_2^{-1 - j} x_1^j Y_0(v,x_2)\sum_{n \in \alpha + \mathbb{Z} } u(n) x_1^{-n-1}\right) \\
&= x_2^{-\alpha} \left( \sum_{j \geq 0} u(\alpha - 1 - j) x_2^j  Y_0(v,x_2) +  Y_0(v,x_2) \sum_{j \geq 0} u(\alpha + j) x_2^{-1 - j} \right) \\
&=   Y_0^+(u,x_2)  Y_0(v,x_2) +  Y_0(v,x_2) Y_0^-(u,x_2) ,
\end{align*}
which is equal to
\begin{align*}
& \res_{x_0} \res_{x_1} x_0^{-1} x_1^{\alpha} x_2^{-\alpha} x_1^{-1} \delta \left( \frac{ x_2 + x_0}{x_1} \right) \left( \frac{x_2}{x_1}\right)^{\alpha} Y_0 \left( Y \left( \left( 1 + \frac{ x_0}{x_2} \right)^{\mathcal{L}} u, x_0 \right) v,x_2 \right)  \\
&= \res_{x_0} x_0^{-1}  Y_0 \left( Y \left( \left( 1 + \frac{ x_0}{x_2} \right)^{\mathcal{L}} u, x_0 \right) v,x_2 \right)  \\
&= \res_{x_0}  x_0^{-1} Y_0 \left( \sum_{n \in \mathbb{Z}} \left( \sum_{j \geq 0}   x_0^j x_2^{-j} \binom{\mathcal{L}}{j} u  \right) (n) x_0^{-n-1} v, x_2 \right) \\
&=  \sum_{j \geq 0}  x_2^{-j} Y_0 \left(  \left( \binom{\mathcal{L}}{j} u  \right) (-1 + j)  v, x_2 \right) . \qedhere
\end{align*}
\end{proof}

All normal orderings of multiple twisted vertex operators will be chosen to be nested as 
\begin{equation}\label{eq:nested norm}
\coln Y_0(u_1,x) \cdots Y_0(u_k,x) a(x) \coln = \coln Y_0(u_1,x)  \coln Y_0(u_2,x) \cdots Y_0(u_{k},x) a(x)\coln \coln ,
\end{equation}
recursively. Observe that \eqref{eq:norm order def} acting on any $ w \in W $ is a lower-truncated element in $ W\{x\} $, making the nested normal ordering \eqref{eq:nested norm} well defined.
\begin{lemma} \label{lem:multi norm order}
Let $ u_1, \dots, u_k \in V $. The normal ordering of multiple twisted vertex operators can be expressed as
\begin{equation}
\label{eq:multi norm order}
\begin{aligned}
&\coln Y_0(u_1,x) \cdots Y_0(u_k,x) \coln \\
& \qquad = \sum_{j_1, \dots, j_{k-1} \geq 0} x^{-j_1 - \cdots - j_{k-1}} Y_0\left( \left( \binom{\mathcal{L}}{j_1}u_1 \right) (-1+j_1) \cdots \left( \binom{\mathcal{L}}{j_{k-1}}u_{k-1} \right)(-1+j_{k-1}) u_k , x \right)  .
\end{aligned}
\end{equation}

\end{lemma}

\begin{proof}
We proceed by induction with Lemma~\ref{lem:norm order} as the base case. Assuming equation \eqref{eq:multi norm order} for some $ k \in \mathbb{Z}_{\geq 0} $, we have
\begin{align*}
&\coln Y_0(u_1,x) Y_0(u_2,x)\cdots  Y_0(u_{k+1},x)\coln \\ &\quad = \coln Y_0(u_1,x) \coln Y_0(u_2,x) \cdots Y_0(u_{k+1},x)\coln \coln \\
&\quad =  \coln Y_0(u_1,x) \sum_{j_2, \dots, j_{k} \geq 0} x^{-j_2 - \cdots - j_{k}} Y_0\left( \left( \binom{\mathcal{L}}{j_2}u_2 \right) (-1+j_2) \cdots \left( \binom{\mathcal{L}}{j_{k}}u_{k} \right)(-1+j_{k}) u_{k+1} , x \right) \coln  \\
&\quad = \sum_{j_1, \dots, j_{k} \geq 0} x^{-j_1 - \cdots - j_{k}} Y_0\left( \left( \binom{\mathcal{L}}{j_1}u_1 \right) (-1+j_1) \cdots \left( \binom{\mathcal{L}}{j_{k}}u_{k} \right)(-1+j_{k}) u_{k+1} , x \right). \qedhere
\end{align*}
\end{proof}

The previous lemma will be useful when we require that $ V $ is $ C_2 $-cofinite and, hence, can deduce certain behavior when $ k $ is sufficiently large. To get stronger finiteness results, we \textit{assume that $ V $ is $ C_2 $-cofinite} from here on. This implies that there is a finite-dimensional subspace $ U $ of $ V $ such that $ V = U + C_2(V) $. Hence, there exists $ M \in \mathbb{Z} $ such that $ v \in V_{(n)} $ with $ n \geq M $ implies that $ v \in C_2(V) $. We choose such an $ M $ (dependent on $ V $ only).

\begin{lemma} \label{lem:repeats}
Let $ N \in \mathbb{Z} $ and $ k \in \mathbb{Z}_{\geq M} $. Let $ u_j \in V_+^{[\alpha_j]} $ for $ j = 1, \dots, k $. Then
\begin{equation}
u_1(\alpha_1 - N) \cdots u_k(\alpha_k - N) \in E^{( \wt u_1 + \cdots + \wt u_k - 1 )} + E^{( \wt u_1 + \cdots + \wt u_k, \, -N - 1 )} .
\end{equation}
\end{lemma}

\begin{proof}
We first consider the case that $ N \in \mathbb{Z}_{>0} $. From Lemma~\ref{lem:multi norm order}, we have
\begin{align*}
&\coln Y_0(u_1,x) \cdots Y_0(u_k,x) \coln \\
& \qquad = \sum_{j_1, \dots, j_{k-1} \geq 0} x^{-j_1 - \cdots - j_{k-1}} Y_0\left( \left( \binom{\mathcal{L}}{j_1}u_1 \right) (-1+j_1) \cdots \left( \binom{\mathcal{L}}{j_{k-1}}u_{k-1} \right)(-1+j_{k-1}) u_k , x \right).
\end{align*}
On the right-hand side, when $ j_1 = \cdots = j_{k-1} = 0 $, the element $ u_1(-1) \cdots u_{k-1}(-1) u_k $ has weight $ \wt u_1 + \cdots + \wt u_k \geq k \geq M $, implying it is in $ C_2(V) $. Application of Lemma~\ref{lem:C2 elements} then shows that the modes that come from this term are indeed in $ E^{( \wt u_1 + \cdots + \wt u_k - 1 )} $. In the case that any of $ j_1 , \dots,  j_{k-1} $ are greater than zero, we have
\begin{align*}
\wt \left( \binom{\mathcal{L}}{j_1} u_1 \right)(-1+j_1) \cdots \left( \binom{\mathcal{L}}{j_k} u_{k-1} \right)(-1+j_{k-1}) u_k &< \wt u_1(-1) \cdots u_{k-1}(-1) u_k  \\
&=  \wt u_1 + \cdots +\wt u_{k} .
\end{align*}
Hence, the modes that come from this term are also in $ E^{( \wt u_1 + \cdots + \wt u_k - 1 )} $.
By the definition of normal ordering, the left-hand side is a sum of terms of the form
\begin{align*}
Y_0^\pm (u_*,x) \cdots Y_0^\pm(u_*,x),
\end{align*}
where $ * $ denotes some number in $ \{1, \dots, k \}$, each being used exactly once. Since $ N > 0 $, $ u_1(\alpha_1 - N) \cdots u_k (\alpha_k - N) $ is a coefficient in the term $ Y_0^+ (u_1,x) \cdots Y_0^+(u_k,x) $, and to extract it, we apply $ \res_{x} x^{\alpha_1 + \cdots + \alpha_k - kN + k - 1 } $ to the left-hand side. Other terms $ u_*(\alpha_* - n_*) \cdots u_* (\alpha_* - n_*) $ with $ n_1 + \dots + n_k = kN $ will be extracted as well. However, if $ (n_1, \dots, n_k) \neq (N, \dots, N) $ but $ n_1 + \dots + n_k = kN $, then there must be some $ n_j \geq  N + 1 $. Hence, these terms are in $ E^{( \wt u_1 + \cdots + \wt u_k, \, -N - 1 )} $. Thus, applying $ \res_{x} x^{\alpha_1 + \cdots + \alpha_k - kN + k - 1 } $ to both sides and rearranging for $ u_1(\alpha_1 - N) \cdots u_k(\alpha_k - N) $ shows the desired result. 

If $ N \leq 0 $, then  $ u_1(\alpha_1 - N) \cdots u_k(\alpha_k - N) $ is a coefficient in $ Y_0^- (u_1,x) \cdots Y_0^-(u_k,x) $, which has the reversed order of multiplication as the term of all singular parts in $ \coln Y_0(u_1,x) \cdots Y_0(u_k,x) \coln $. To obtain the desired result, we repeat similar arguments as above but for $ \coln Y_0(u_k,x) \cdots Y_0(u_1,x) \coln $.
\end{proof}

We are now ready to show that the $ C_2 $-cofiniteness of $ V $ is strong enough to impose the $ C_n $-cofiniteness on every finitely generated weak $ g $-twisted $ V $-module. The main idea will be to find a spanning set for such a module $ W $ with only finitely many elements not in $ C_n(W) $.

Since $ C_2(V) = \Span \{ u(-2) v : u \in V, \ v \in V \} $ is equivalent to  
\begin{align*}
C_2(V) = \Span \left\{ u(-2) v : u \in V_{(m)}^{[\alpha]}, \ v \in V_{(n)}^{[\beta]}, \ m,n \in \mathbb{Z}, \ \alpha, \beta \in P(V) \right\},
\end{align*} 
we see that $ C_2(V) $ has a grading by weight and $ \mathcal{S} $-eigenvalues similar to that of $ V $. Hence, $ V/C_2(V) $ inherits this grading. Let $ \{ [b_i] \}_{i \in I} $ be a basis for $ V/C_2(V) $ consisting of homogeneous vectors. Let $ \{ b_i \}_{i \in I} $ be a set of representatives for this basis (also consisting of homogeneous vectors).  Since $ V $ is assumed to be $ C_2 $-cofinite, we see that $ B $ is a finite set. 

We \emph{assume $ V $ is of CFT type} from here on. Since $ V $ is of CFT type, there is some non-zero multiple of $ \one $ in $ \{ b_i \}_{i \in I} $, and no other weight-zero vectors or negative-weight vectors in $ \{ b_i \}_{i \in I} $. For convenience, we will remove this vector from $ \{ b_i \}_{i \in I} $ and call this new set $ B $. Note that every vector $ b_i \in B $ has $ \wt b_i \geq 1 $. We will denote the $ \mathcal{S} $-eigenvalue for $ b_i $ by $ \alpha_i $.

To define the following filtrations, we \textit{assume that $ W $ is generated by an element $ \w $}. To prove the main proposition below, we will use the filtration
\begin{equation}
W^{(0)} \subseteq W^{(1)} \subseteq  \cdots \subseteq W 
\end{equation}
of $ W $ defined by $ W^{(s)} = E^{(s)} \w $, and the filtration
\begin{equation}
 \cdots \subseteq W^{(s,-1)}  \subseteq W^{(s,0)} \subseteq W^{(s,1)} \subseteq \cdots \subseteq W^{(s)} ,
\end{equation}
of $ W^{(s)} $ defined by $ W^{(s,t)} = E^{(s,t)} \w $. That is, any element $ w \in W^{(s)} $ is a sum of elements of the form \begin{equation} \label{eq:W(s) element}
u_1(\alpha_1 - n_1) \cdots u_k(\alpha_k - n_k) \w
\end{equation}
with $ k \in \mathbb{Z}_{>0} $,  $ n_j \in \mathbb{Z} $, and $ u_j \in V^{[\alpha_j]} $ satisfying $ \sum_{j=1}^k \wt {u_j} \leq s $. Furthermore, $ W^{(s,t)} $ is spanned by similar elements, but also with the condition that there is some $ n_i \geq  - t $. Since $ \w $ generates $ W $, we have $ W = \bigcup_{s \in \mathbb{Z}_{\geq 0}} W^{(s)} $.

Since $ B $ is finite and $ Y_W $ satisfies the lower truncation property, there exists some $ L \in \mathbb{Z} $ such that $ b_i(\alpha_i - n) \w  = 0 $ for all $ n \leq L $ and for all $ b_i \in B $. We choose such an $ L $ (dependent on $ B $, $ W $ and $ \w $). Recall that we also have $ M \in \mathbb{Z}_{\geq 0} $ such that $ v \in \coprod_{n \geq M } V_{(n)} $ implies $ v \in C_2(V) $. 

\begin{proposition}\label{prop:C2 spanning set}
Let $ V $ be a $ C_2 $-cofinite vertex operator algebra of CFT type with an automorphism $ g $. Let $ W $ be a weak $ g $-twisted $ V $-module generated by a vector $ \w $. Let $ B $, $ L $ and $ M $ be defined as above. Let $  N \in \mathbb{Z}_{\geq L} $. Then $ W $ is spanned by elements of the form
\begin{equation}\label{eq:spanning set element}
b_{i_1} (\alpha_{i_1} - n_1) \cdots b_{i_k} (\alpha_{i_k} - n_k) \w 
\end{equation}
with $ k \in \mathbb{Z}_{\geq 0 }$, $ b_{i_1} , \dots, b_{i_k} \in B $ and $ n_1 , \dots,  n_k \in \mathbb{Z}_{> L} $ satisfying
\begin{enumerate}
\item[(i)] (ordering condition) \quad $ n_1 \geq \cdots \geq n_k $; and
\item[(ii)](repeat condition) \qquad $ n_j = n_{j+1} = \cdots = n_{j + M - 2} \leq N $ implies $ n_{j + M - 2} \neq n_{j + M-1} $ (i.e. $ n_j $ cannot repeat $ M $ times when $ n_j \leq N $).
\end{enumerate}
\end{proposition}

\begin{proof}
For each $ N \in \mathbb{Z}_{\geq L} $, define $ W^{[N]} $ to be the subspace of $ W $ spanned by elements of the form \eqref{eq:spanning set element} satisfying the ordering and repeat conditions. For each $ (s,N) \in \mathbb{Z}_{\geq 0} \times \mathbb{Z}_{\geq L} $, consider the statement $ P(s,N) $: $ W^{(s)} \subseteq W^{[N]}$.  We prove this is true for all $ (s,N) \in \mathbb{Z}_{\geq 0} \times \mathbb{Z}_{\geq L} $ by induction.

Since $ V $ is of CFT type, we see that $ W^{(0)} = \mathbb{C} \w $, hence the base cases $ P(0, N) $ for all $ N \geq L $ are immediately true. 

Suppose that the case $ P(s-1,L) $ is true for some $ s > 0 $. Consider  an element in $ W^{(s)} $ of the form \eqref{eq:W(s) element}.  Each $ u_j $ can be expressed in terms of elements in $ B $ or $ C_2(V) $ (or possibly $ \one $, but this case can be immediately ignored by the identity property). The components in $ C_2(V) $ give elements in $ W^{(s-1)} $ by Lemma~\ref{lem:C2 elements}. The remaining elements are of the form \eqref{eq:spanning set element}, possibly with the wrong ordering of the modes. Lemma~\ref{lem:swap} allows this reordering up to an element in $ W^{(s-1)} $. After ordering, any element with some $ b_i(\alpha_i - n) $ with $ n \leq L $ will be zero. Since $ N = L $, there is no repeat condition to check. Hence, $ w $ is a sum of elements of the form \eqref{eq:spanning set element} satisfying the ordering and repeat conditions, and elements in $ W^{(s-1)} $. The latter terms are handled by the induction hypothesis $ P(s-1,L) $. Hence, $ P(s,L) $ is true. 

Suppose that, for some $ (s,N) \in \mathbb{Z}_{>0} \times \mathbb{Z}_{>L} $, $ P(t,q) $ is true for all $ t \in \{ 0, \dots, s \} $ and $ q \in \{ L, \dots, N \} $ with $ t < s $ or $ q < N $. By $ P(s,N-1) $, any element in $ W^{(s)} $ can be reduced to a sum of elements of the form
\begin{align*}
b_{i_1} ( \alpha_{i_1} - n_1) \cdots  b_{i_{k}}( \alpha_{i_k} - n_k) \w,
\end{align*}
such that $ n_1 \geq \cdots \geq n_k \in \mathbb{Z}_{>L} $ and $ n_j = n_{j+1} = \cdots = n_{j + M - 2}  \leq N - 1 $ implies that $ n_{j + M - 2} \neq n_{j + M - 1} $. Consider an element of this form. If less than $ M $ of the $ n_j $ are equal to $ N $, then we already have the desired form of an element in $ W^{[N]} $. Otherwise we have \begin{align*}
n_1 \geq \cdots \geq n_{p-1} > N = n_p = \cdots = n_{p + m-1} > n_{p + m} \geq \cdots \geq n_k > L
\end{align*} 
for some $ p \geq 1 $ and $ m \geq M $. 

If $ p > 1 $, then $ b_{i_2}(\alpha_{i_2} - n_2) \cdots  b_{i_k}(\alpha_{i_k} - n_k) w \in W^{(s-\wt{b_{i_1}})} $, which is in $ W^{[N]} $ by $ P(s-\wt{b_{i_1}}, N) $. Replacing $ b_{i_1}(\alpha_{i_1} - n_1) $, we get elements of the form \eqref{eq:spanning set element}, but possibly with $ b_{i_1}(\alpha_{i_1} - n_1) $ out of place. We can move $ b_{i_1}(\alpha_{i_1} - n_1) $ into its correct place (up to elements in $ W^{(s-1)} \subseteq W^{[N]} $) with Lemma~\ref{lem:swap}. Since $ n_1 > N $, moving $ b_{i_1}(\alpha_{i_1} - n_1) $ does not break the repeat condition, hence we obtain an element in $ W^{[N]} $. 

If $ p = 1 $, then we use Lemma~\ref{lem:repeats} to replace $ b_{i_1} ( \alpha_{i_1} - N) \cdots  b_{i_{m}}( \alpha_{i_{m}} - N) $, obtaining elements in $ W^{(s-1)} $ or $ W^{(s,-N-1)} $. The elements in $ W^{(s,-N-1)} $ can be expressed in terms of $ B $ or elements in $ C_2(V) $ and reordered to obtain the case $ p > 1 $, which we have already resolved. Hence, $ P(s,N) $ is true.

So, for each fixed $ N \in \mathbb{Z}_{\geq L} $, we have $ W^{(s)} \subseteq W^{[N]} $ for all $ s \in \mathbb{Z}_{\geq 0} $. Thus, $ W \subseteq W^{[N]} $.
\end{proof}

\begin{remark}
We could further refine the spanning set in Proposition~\ref{prop:C2 spanning set} to include the condition $ n_j \neq n_{j+1} $ whenever $ n_{j+1} \geq 2 $. To obtain this, we would prove \begin{align*}
u(\alpha -N)v(\beta - N) &= \sum_{k \geq 0} \binom{\alpha}{k} (u(-1 +k) v)(\alpha + \beta - 2N + 1 - k)  \\
&- \sum_{\substack{k \geq 0 \\ k \neq N-1}} u(\alpha -1 -k)v( \beta - 2N + 1 + k ) - \sum_{k \geq 0} v(\beta - 2N -k) u(\alpha + k)  
\end{align*}  
when $ N > 0 $, and is in  $ E^{(\wt u + \wt v - 1)} + E^{(\wt u + \wt v,\, -N - 1)} 
$ when $ N > 2 $.
This refined spanning set will not be needed to prove $ C_n $-cofiniteness.
\end{remark}

\begin{theorem}\label{thm:Cn cofiniteness}
Let $ V $ be a $ C_2 $-cofinite CFT-type vertex operator algebra with an automorphism $ g $. Let $ W $ be a finitely generated weak $ g $-twisted $ V $-module. Then $ W $ is $ C_n $-cofinite for all $ n \in \mathbb{Z}_{>0} $.
\end{theorem}

\begin{proof}
Let $ W $ be generated by $ \{\w_1, \dots, \w_k \} $. Proposition~\ref{prop:C2 spanning set} gives a spanning set for each submodule of $ W $ generated by $ \w_j $. The union of these spanning sets gives a spanning set for $ W $. Since $ B $ is finite, we can pick $ N $ in Proposition~\ref{prop:C2 spanning set} to be sufficiently large so that there are only finitely many elements in the spanning set of the form
\begin{align*}
b_{i_1} ( \alpha_{i_1} - n_1) \cdots  b_{i_{k}}( \alpha_{i_k} - n_k) \w_j
\end{align*} 
with $ n_1 < n $. The rest must have $ n_1 \geq n > 0 $. When $ n_1 = n + j $ for some $ j \geq 1 $, we can express $ b_{i_1} ( \alpha_{i_1} - n_1) $ as a sum of modes of the form $ v_j ( \alpha_{i_1} - n) $ for some $ v_j \in V_+ $, by Lemma~\ref{lem:lower modes}. Hence, all but a finite set of elements in the spanning set are in $ C_n(W) $. Hence, $ W/ C_n(W) $ is finite dimensional. 
\end{proof}

\printbibliography[category=cited]

\noindent {\small \sc Department of Mathematics, Rutgers University,
110 Frelinghuysen Rd., Piscataway, NJ 08854-8019}

\noindent {\em E-mail address}: dwt24@scarletmail.rutgers.edu

\end{document}